\documentclass[12pt]{article}%
\usepackage{amssymb,latexsym}
\usepackage{amsfonts}
\usepackage{amsmath}
\usepackage[colorlinks=true, pdfstartview=FitV, linkcolor=blue,
citecolor=blue, urlcolor=blue]{hyperref}
\usepackage{color}
\usepackage{amssymb}
\usepackage{graphicx}%

\newtheorem{theorem}{Theorem}

\newtheorem{definition}[theorem]{Definition}

\newtheorem{proposition}[theorem]{Proposition}
\newtheorem{remark}[theorem]{Remark}

\newenvironment{proof}[1][Proof]{\noindent\textbf{#1.} }{\ \rule{0.5em}{0.5em}}
\numberwithin{theorem}{section}
\allowdisplaybreaks
\begin{document}

\title{Character analogue of the Boole summation formula with applications}
\author{M\"{u}m\"{u}n Can and M. Cihat Da\u{g}l\i\ \\Department of Mathematics, Akdeniz University,\\07058-Antalya, Turkey 
\\mcan@akdeniz.edu.tr,\ \ mcihatdagli@akdeniz.edu.tr}
\date{}
\maketitle

\begin{abstract}
In this paper, we present the character analogue of the Boole summation
formula. Using this formula, an integral representation is derived for the
alternating Dirichlet $L-$function and its derivative is evaluated at $s=0$.
Some applications of the character analogue of the Boole summation formula and
the integral representation are given about the alternating Dirichlet
$L-$function. Moreover, the reciprocity formulas for two new arithmetic sums,
arising from utilizing the summation formulas, and for Hardy--Berndt sum
$S_{p}\left(  b,c:\chi\right)  $ are proved.



\end{abstract}

\section{Introduction}

The Euler--MacLaurin summation formula is a well-known formula from classical
analysis giving a relation between the finite sum of values of a function and
its integral. One of the generalizations of the Euler--MacLaurin summation
formula is the character analogue due to Berndt \cite{2} which is presented
here in the following form.

\begin{theorem}
\label{CEM}(\cite[Theorem 4.1]{2}) Let $\chi$ be a primitive character of
modulus $k$ with $k>1.$ For $f\in C^{(l+1)}\left[  \alpha,\beta\right]  $,
$-\infty<\alpha<\beta<\infty$%
\begin{align*}
\sum_{\alpha\leq n\leq\beta}\hspace{-0.05in}^{^{\prime}}\chi\left(  n\right)
f(n)  &  =\chi\left(  -1\right)  \sum\limits_{j=0}^{l}\frac{\left(  -1\right)
^{j+1}}{(j+1)!}\left(  \overline{B}_{j+1,\overline{\chi}}\left(  \beta\right)
f^{(j)}(\beta)-\overline{B}_{j+1,\overline{\chi}}(\alpha)f^{(j)}%
(\alpha)^{\text{\ }}\right) \\
&  \quad+\chi\left(  -1\right)  \frac{(-1)^{l}}{(l+1)!}\int\limits_{\alpha
}^{\beta}\overline{B}_{l+1,\overline{\chi}}\left(  u\right)  f^{(l+1)}(u)du,
\end{align*}
where the dash indicates that if $n=\alpha$ or $n=\beta$, then only $\frac
{1}{2}\chi(\alpha)f(\alpha)$ or $\frac{1}{2}\chi(\beta)f(\beta)$ is counted,
respectively. Also, $\overline{B}_{p,\chi}\left(  x\right)  $ denotes the
generalized Bernoulli function defined by (\ref{13}).
\end{theorem}

The alternating version of the Euler--MacLaurin summation formula is the Boole
summation formula (\cite[24.17.1--2]{20}), as pointed out by N\"{o}rlund
\cite{18}, this formula is also due to Euler.

\begin{theorem}
\label{BS}(Boole summation formula) For integers $\alpha,\beta$ and $l$\ such
that $\alpha<\beta$ and $l>0$%
\begin{align*}
2\sum_{n=\alpha}^{\beta-1}\left(  -1\right)  ^{n}f(n)  &  =\sum\limits_{j=0}%
^{l-1}\frac{E_{j}\left(  0\right)  }{j!}\left(  \left(  -1\right)  ^{\beta
-1}f^{(j)}(\beta)+\left(  -1\right)  ^{\alpha}f^{(j)}(\alpha)^{\text{\ }%
}\right) \\
&  \quad+\frac{1}{(l-1)!}\int\limits_{\alpha}^{\beta}f^{(l)}(x)\overline
{E}_{l-1}\left(  -x\right)  dx,
\end{align*}
where $f^{\left(  l\right)  }(x)$ is absolutely integrable over $\left[
\alpha,\beta\right]  ,$ and $\overline{E}_{p}\left(  x\right)  $ is the Euler
function defined by (\ref{18}).
\end{theorem}

By the authors' knowledge, the character generalization of the Boole summation
formula is not available.

In this paper, we first present the character analogue of the Boole summation
formula as

\begin{theorem}
\label{C-B-S}Let $\chi$ be a primitive character of modulus $k$ with $k>1$
odd. If $f\in C^{(l+1)}\left[  \alpha,\beta\right]  ,$ $-\infty<\alpha
<\beta<\infty,$ then%
\begin{align*}
2\sum_{\alpha<n<\beta}\left(  -1\right)  ^{n}\chi\left(  n\right)  f(n)  &
=\chi\left(  -1\right)  \sum\limits_{j=0}^{l}\frac{\left(  -1\right)  ^{j}%
}{j!}\left(  \overline{E}_{j,\overline{\chi}}\left(  \beta\right)
f^{(j)}(\beta)-\overline{E}_{j,\overline{\chi}}(\alpha)f^{(j)}(\alpha
)^{\text{\ }}\right) \\
&  \quad-\chi\left(  -1\right)  \frac{(-1)^{l}}{l!}\int\limits_{\alpha}%
^{\beta}\overline{E}_{l,\overline{\chi}}\left(  t\right)  f^{(l+1)}(t)dt,
\end{align*}
where $\overline{E}_{l,\chi}\left(  t\right)  $ is the generalized Euler
function defined by (\ref{13a}).
\end{theorem}

Later, we give applications of this formula on two subjects. The first one is
around the alternating Dirichlet $L$--function. For $a\not =-1,-2,-3,\ldots$,
let $\ell\left(  s,a,\chi\right)  $ denote the alternating Dirichlet
$L$--function
\[
\ell\left(  s,a,\chi\right)  =\sum_{n=1}^{\infty}\left(  -1\right)  ^{n}%
\frac{\chi\left(  n\right)  }{\left(  n+a\right)  ^{s}},\text{ }%
\operatorname{Re}\left(  s\right)  >0,
\]
which can be written in terms of Hurwitz zeta function $\zeta\left(
s,a\right)  $ as
\[
\ell\left(  s,a,\chi\right)  =\left(  2k\right)  ^{-s}\sum_{j=1}^{2k-1}\left(
-1\right)  ^{j}\chi\left(  j\right)  \zeta\left(  s,\frac{a+j}{2k}\right)
\]
for $\operatorname{Re}\left(  s\right)  >1.$ Also let
\[
\ell_{s}\left(  x,a,\chi\right)  =\sum_{1\leq n\leq x}\left(  -1\right)
^{n}\chi\left(  n\right)  \left(  n+a\right)  ^{s},\text{ }x\geq0.
\]
Then, the integral representations for $\ell\left(  s,a,\chi\right)  $\ and
$\ell_{s}\left(  x,a,\chi\right)  $ are derived as in the following.

\begin{theorem}
\label{int}Let $\chi$ be a primitive character of modulus $k$ with $k>1$ odd.
For $l\geq0$ with $l>\operatorname{Re}\left(  s\right)  $ and for any
$x\geq0,$ we have the integral representation\textbf{\ }%
\begin{align*}
2\ell_{s}\left(  x,a,\chi\right)   &  =\chi\left(  -1\right)  \sum
\limits_{j=0}^{l}\left(  -1\right)  ^{j}\frac{\left(  s\right)  _{j}}%
{j!}\overline{E}_{j,\overline{\chi}}\left(  x\right)  \left(  x+a\right)
^{s-j}+2\ell\left(  -s,a,\chi\right) \\
&  \quad-\frac{\left(  s\right)  _{l+1}}{l!}\int\limits_{x}^{\infty}%
\overline{E}_{l,\overline{\chi}}\left(  -t\right)  \left(  t+a\right)
^{s-l-1}dt,
\end{align*}
where $\left(  s\right)  _{j}=s\left(  s-1\right)  \cdots\left(  s-j+1\right)
$ with $\left(  s\right)  _{0}=1.$

Moreover, for $x=0$ we have
\[
2\ell\left(  -s,a,\chi\right)  =\sum\limits_{j=0}^{l}\frac{\left(  s\right)
_{j}}{j!}E_{j,\overline{\chi}}\left(  0\right)  a^{s-j}+\frac{\left(
s\right)  _{l+1}}{l!}\int\limits_{0}^{\infty}\overline{E}_{l,\overline{\chi}%
}\left(  -t\right)  \left(  t+a\right)  ^{s-l-1}dt.
\]

\end{theorem}

Besides, some formulas, such as character analogue of the Lerch's formula for
Hurwitz zeta function (see (\ref{35})), character analogues of the Stirling's
formula for $\log\Gamma\left(  a\right)  $ and of the Weierstrass product for
$\Gamma\left(  a\right)  $ (see Propositions \ref{Stf} and \ref{GEP} below,
respectively) are deduced via Theorems \ref{C-B-S} and \ref{int}.

The second is around the Hardy--Berndt sums. Let us mention that utilizing the
summation formulas, alternative proofs of the reciprocity formulas of certain
Dedekind sums and their analogues have been offered in \cite{29,5,10,11,14}.
Here, we reveal that new arithmetic sums obeying reciprocity law may be
defined by the summation formulas aforementioned above. We describe two of
such sums as%
\begin{align*}
S_{p}^{\left(  1\right)  }\left(  b,c:\chi\right)   &  =\sum\limits_{n=1}%
^{ck}\left(  -1\right)  ^{n}\overline{E}_{p,\overline{\chi}}\left(  \frac
{bn}{c}\right)  ,\\
S_{p}^{\left(  2\right)  }\left(  b,c:\chi\right)   &  =\sum\limits_{n=1}%
^{ck}\left(  -1\right)  ^{n}\chi\left(  n\right)  \overline{E}_{p}\left(
\frac{bn}{c}\right)
\end{align*}
and prove the following reciprocity formula.

\begin{theorem}
\label{recip2}Let $b$ and $c$ be positive integers with $(b+c)$ odd and
$\chi\left(  -1\right)  \left(  -1\right)  ^{p}=1$. Then, the reciprocity
formula holds:
\[
c^{p}S_{p}^{\left(  1\right)  }\left(  b,c:\overline{\chi}\right)  +b^{p}%
S_{p}^{\left(  2\right)  }\left(  c,b:\chi\right)  =2\sum\limits_{j=0}%
^{p}\binom{p}{j}c^{j}b^{p-j}\overline{E}_{j,\overline{\chi}}\left(  0\right)
\overline{E}_{p-j}\left(  0\right)  .
\]

\end{theorem}

In fact, these sums are generalizations of\ one of the Hardy--Berndt sums
\cite{7}%
\[
S\left(  b,c\right)  =\sum_{n=1}^{c-1}\left(  -1\right)  ^{n+1+\left[
bn/c\right]  }.
\]
For various generalizations and properties of Hardy--Berndt sums, the reader
may consult to \cite{7,3,12,29,9,19,13,22,16,17,21,23,24,25,26} and
\cite{25,26} for association between $S\left(  b,c\right)  $\ and Dirichlet
$L$--function. One of the generalizations of $S\left(  b,c\right)  $ has been
given by \cite{29}
\[
S_{p}\left(  b,c:\chi\right)  =\sum\limits_{n=1}^{ck}\chi\left(  n\right)
\overline{B}_{p,\overline{\chi}}\left(  \frac{b+ck}{2c}n\right)
\]
and the corresponding reciprocity formula is proved via transformation
formulas. Here, utilizing Theorem \ref{C-B-S}, we give a new proof for the
following reciprocity formula by refining the conditions.

\begin{theorem}
\label{recip1}Let $p>1$ be odd and let $b$ and $c$ be positive integers with
$(b+c)$ odd. Then, the reciprocity formula holds:%
\begin{align*}
&  \overline{\chi}\left(  -2\right)  bc^{p}S_{p}\left(  b,c:\chi\right)
+\chi\left(  -2\right)  cb^{p}S_{p}\left(  c,b:\overline{\chi}\right) \\
&  =\frac{p}{2^{p+1}}\sum\limits_{j=1}^{p}\left(  -1\right)  ^{j}\binom
{p-1}{j-1}c^{j}b^{p+1-j}\overline{E}_{j-1,\chi}\left(  0\right)  \overline
{E}_{p-j,\overline{\chi}}\left(  0\right)  .
\end{align*}

\end{theorem}

The remainder of this paper is organized as follows: Section 2 is the
preliminary section where we give definitions and known results needed. In
Section 3, we prove Theorem \ref{C-B-S} and Theorem \ref{int}. Some
applications of the integral representation and the character analogue of the
Boole summation formula are given in Section 4. The final section is devoted
to prove the reciprocity formulas for the Hardy--Berndt sums mentioned above
via summation formulas.

\section{Preliminaries}

The Bernoulli and Euler polynomials $B_{n}(x)$ and $E_{n}(x)$\ are defined by
means of the generating functions \cite{1}
\begin{align*}
\frac{te^{xt}}{e^{t}-1}  &  =\sum\limits_{n=0}^{\infty}B_{n}(x)\frac{t^{n}%
}{n!},\quad\left(  \left\vert t\right\vert <2\pi\right)  ,\\
\frac{2e^{xt}}{e^{t}+1}  &  =\sum\limits_{n=0}^{\infty}E_{n}(x)\frac{t^{n}%
}{n!},\quad\left(  \left\vert t\right\vert <\pi\right)  ,
\end{align*}
respectively. In particular, the rational numbers $B_{n}=B_{n}(0)$ and
integers $E_{n}=2^{n}E_{n}(1/2)$ are called classical Bernoulli numbers and
Euler numbers, respectively.

For $0\leq x<1$ and $m\in\mathbb{Z}$, the Bernoulli functions $\overline
{B}_{n}\left(  x\right)  $ are defined by
\[
\overline{B}_{n}\left(  x+m\right)  =B_{n}\left(  x\right)  \text{ when
}n\not =1\text{ and }x\not =0,\text{ and }\overline{B}_{1}\left(  m\right)
=\overline{B}_{1}\left(  0\right)  =0
\]
and the Euler functions $\overline{E}_{n}\left(  x\right)  $ are defined by
\cite{6}
\begin{equation}
\overline{E}_{n}\left(  x+m\right)  =\left(  -1\right)  ^{m}\overline{E}%
_{n}\left(  x\right)  \text{ and }\overline{E}_{n}\left(  x\right)  =E_{n}(x).
\label{18}%
\end{equation}
The Bernoulli functions satisfy the Raabe or multiplication formula%
\[
r^{n-1}\sum_{j=0}^{r-1}\overline{B}_{n}\left(  x+\frac{j}{r}\right)
=\overline{B}_{n}\left(  rx\right)
\]
and also following property is valid for even $r$%
\begin{equation}
r^{n-1}\sum_{j=0}^{r-1}\left(  -1\right)  ^{j}\overline{B}_{n}\left(
\frac{x+j}{r}\right)  =-\frac{n}{2}\overline{E}_{n-1}\left(  x\right)  .
\label{26}%
\end{equation}

$\overline{B}_{m,\chi}\left(  x\right)  $ denotes the generalized Bernoulli
function, with period $k,$ defined by Berndt \cite{2}.\ We will often use the
following property that can confer as a definition
\begin{equation}
\overline{B}_{m,\chi}\left(  x\right)  =k^{m-1}\sum_{j=0}^{k-1}\overline{\chi
}\left(  j\right)  \overline{B}_{m}\left(  \frac{j+x}{k}\right)  ,\ m\geq1.
\label{13}%
\end{equation}
For the convenience with the definition of $\overline{B}_{m,\chi}\left(
x\right)  ,$ let the character Euler function $\overline{E}_{m,\chi}\left(
x\right)  $ be defined by
\begin{equation}
\overline{E}_{m,\chi}\left(  x\right)  =k^{m}\sum_{j=0}^{k-1}\left(
-1\right)  ^{j}\overline{\chi}\left(  j\right)  \overline{E}_{m}\left(
\frac{j+x}{k}\right)  ,\ m\geq0 \label{13a}%
\end{equation}
for odd $k,$ the modulus of $\chi.$

List some properties that we need in the sequel%
\begin{align}
&  \frac{d}{dx}\overline{E}_{m}\left(  x\right)  =m\overline{E}_{m-1}\left(
x\right)  ,\text{ }m>1,\label{8a}\\
&  \frac{d}{dx}\overline{E}_{m,\chi}\left(  x\right)  =m\overline{E}%
_{m-1,\chi}\left(  x\right)  ,\text{ }m\geq1,\label{8b}\\
&  \overline{E}_{m,\chi}\left(  -x\right)  =\left(  -1\right)  ^{m-1}%
\chi\left(  -1\right)  \overline{E}_{m,\chi}\left(  x\right)  ,\label{8c}\\
&  \overline{E}_{m,\chi}\left(  x+nk\right)  =\left(  -1\right)  ^{n}%
\overline{E}_{m,\chi}\left(  x\right)  . \label{8d}%
\end{align}

In the sequel, unless otherwise stated, we assume that $\chi$ is a primitive
character of modulus $k$ with $k>1$ odd.

\section{Proofs of Theorems \ref{C-B-S} and \ref{int}}

\subsection{Proof of Theorem \ref{C-B-S}}

Firstly, we write%
\begin{align*}
\sum_{\alpha<n<\beta}\left(  -1\right)  ^{n}\chi\left(  n\right)  f(n)  &
=\sum_{\alpha<n<\beta}\chi\left(  2n\right)  f(2n)-\sum_{\alpha<2n+1<\beta
}\chi\left(  2n+1\right)  f(2n+1)\\
&  =2\chi\left(  2\right)  \sum_{\alpha/2<n<\beta/2}\chi\left(  n\right)
f(2n)-\sum_{\alpha<n<\beta}\chi\left(  n\right)  f(n).
\end{align*}
Applying Theorem \ref{CEM} on the right-hand side, one has%
\begin{align}
&  \sum_{\alpha<n<\beta}\left(  -1\right)  ^{n}\chi\left(  n\right)
f(n)\nonumber\\
&  =\chi\left(  -1\right)  \sum\limits_{j=0}^{l}\frac{\left(  -1\right)
^{j+1}}{(j+1)!}\left(  \left(  2^{j+1}\chi\left(  2\right)  \overline
{B}_{j+1,\overline{\chi}}\left(  \frac{\beta}{2}\right)  -\overline
{B}_{j+1,\overline{\chi}}\left(  \beta\right)  \right)  f^{(j)}(\beta)\right.
\label{16}\\
&  \qquad\qquad\qquad\qquad\qquad\left.  -\left(  2^{j+1}\chi\left(  2\right)
\overline{B}_{j+1,\overline{\chi}}\left(  \frac{\alpha}{2}\right)
-\overline{B}_{j+1,\overline{\chi}}(\alpha)\right)  f^{(j)}(\alpha)\right)
\nonumber\\
&  \quad+\chi\left(  -1\right)  \frac{(-1)^{l}}{(l+1)!}\int\limits_{\alpha
}^{\beta}\left(  2^{l+1}\chi\left(  2\right)  \overline{B}_{l+1,\overline
{\chi}}\left(  \frac{u}{2}\right)  -\overline{B}_{l+1,\overline{\chi}}\left(
u\right)  \right)  f^{(l+1)}(u)du.\nonumber
\end{align}
On the other hand, taking $x\rightarrow x/2$ in \cite[Eq. (3.13)]{29} gives
\begin{equation}
\overline{B}_{m,\chi}\left(  \frac{x}{2}\right)  +\overline{B}_{m,\chi}\left(
\frac{x+k}{2}\right)  =2^{1-m}\chi\left(  2\right)  \overline{B}_{m,\chi
}\left(  x\right)  . \label{14}%
\end{equation}
By using (\ref{13}) and (\ref{26}) for $r=2$ we can write%
\begin{equation}
\overline{B}_{m,\chi}\left(  \frac{x}{2}\right)  -\overline{B}_{m,\chi}\left(
\frac{x+k}{2}\right)  =-\frac{m}{2^{m}}k^{m-1}\sum\limits_{v=0}^{k-1}%
\overline{\chi}\left(  v\right)  \overline{E}_{m-1}\left(  \frac{2v+x}%
{k}\right)  . \label{11}%
\end{equation}
Employing basic manipulations, (\ref{11}) becomes
\begin{align}
&  k^{m-1}\left\{  \sum\limits_{v=0}^{\frac{k-1}{2}}\overline{\chi}\left(
2v\right)  \overline{E}_{m-1}\left(  \frac{2v+x}{k}\right)  +\sum
\limits_{v=\frac{k+1}{2}}^{k-1}\overline{\chi}\left(  2v\right)  \overline
{E}_{m-1}\left(  \frac{2v+x}{k}\right)  \right\} \nonumber\\
&  \ =k^{m-1}\left\{  \sum\limits_{v=0}^{\frac{k-1}{2}}\overline{\chi}\left(
2v\right)  \overline{E}_{m-1}\left(  \frac{2v+x}{k}\right)  -\sum
\limits_{v=0}^{\frac{k-3}{2}}\overline{\chi}\left(  2v+1\right)  \overline
{E}_{m-1}\left(  \frac{2v+1+x}{k}\right)  \right\} \nonumber\\
&  \ =k^{m-1}\sum\limits_{v=0}^{k-1}\left(  -1\right)  ^{v}\overline{\chi
}\left(  v\right)  \overline{E}_{m-1}\left(  \frac{v+x}{k}\right) \nonumber\\
&  \ =\overline{E}_{m-1,\chi}\left(  x\right)  , \label{15}%
\end{align}
where we have used (\ref{18}) and (\ref{13a}). Thus, combining (\ref{14}) and
(\ref{15}) leads to
\begin{equation}
2^{m}\chi\left(  2\right)  \overline{B}_{m,\overline{\chi}}\left(  \frac{x}%
{2}\right)  -\overline{B}_{m,\overline{\chi}}\left(  x\right)  =-\frac{m}%
{2}\overline{E}_{m-1,\overline{\chi}}\left(  x\right)  . \label{31}%
\end{equation}
Substituting (\ref{31}) in (\ref{16}) completes the proof.

\subsection{Proof of Theorem \ref{int}}

The method used here have already been employed by Kanemitsu et al. \cite{15}
to the Euler--MacLaurin summation formula to obtain integral representations
for Hurwitz zeta function and its partial sum.

For $\alpha=0$ and $\beta=x,$ let $f(t)=\left(  t+a\right)  ^{s}$ in Theorem
\ref{C-B-S}. Then, from (\ref{8c}),%
\begin{align}
2\ell_{s}\left(  x,a,\chi\right)   &  =2\sum_{0\leq n\leq x}\hspace
{-0.05in}^{^{\prime}}\left(  -1\right)  ^{n}\chi\left(  n\right)  \left(
n+a\right)  ^{s}\nonumber\\
&  =\chi\left(  -1\right)  \sum\limits_{j=0}^{l}\left(  -1\right)  ^{j}%
\frac{(s)_{j}}{j!}\left(  \overline{E}_{j,\overline{\chi}}\left(  x\right)
\left(  x+a\right)  ^{s-j}-\overline{E}_{j,\overline{\chi}}(0)a^{s-j}\right)
\nonumber\\
&  \quad+\frac{(s)_{l+1}}{l!}\int\limits_{0}^{x}\overline{E}_{l,\overline
{\chi}}\left(  -t\right)  \left(  t+a\right)  ^{s-l-1}dt.\nonumber
\end{align}
Since%
\[
\left\vert \overline{E}_{l,\chi}\left(  t\right)  \right\vert \leq4\frac
{l!}{\left(  \pi/k\right)  ^{l+1}}\zeta\left(  l+1\right)  ,\text{ }l\geq1,
\]
the integral
\[
\int\limits_{0}^{\infty}\overline{E}_{l,\overline{\chi}}\left(  -t\right)
\left(  t+a\right)  ^{s-l-1}dt
\]
is absolutely convergent for $\operatorname{Re}(s)<l.$ So, we may write%
\begin{align}
&  2\ell_{s}\left(  x,a,\chi\right) \nonumber\\
&  =\chi\left(  -1\right)  \sum\limits_{j=0}^{l}\left(  -1\right)  ^{j}%
\frac{(s)_{j}}{j!}\overline{E}_{j,\overline{\chi}}\left(  x\right)  \left(
x+a\right)  ^{s-j}-\chi\left(  -1\right)  \sum\limits_{j=0}^{l}\left(
-1\right)  ^{j}\frac{(s)_{j}}{j!}E_{j,\overline{\chi}}(0)a^{s-j\text{\ }%
}\nonumber\\
&  \quad+\frac{(s)_{l+1}}{l!}\int\limits_{0}^{\infty}\overline{E}%
_{l,\overline{\chi}}\left(  -t\right)  \left(  t+a\right)  ^{s-l-1}%
dt-\frac{(s)_{l+1}}{l!}\int\limits_{x}^{\infty}\overline{E}_{l,\overline{\chi
}}\left(  -t\right)  \left(  t+a\right)  ^{s-l-1}dt. \label{23}%
\end{align}

Now, for $\operatorname{Re}(s)<0,$ letting $x$ tends to $\infty$ in
(\ref{23}), we arrive at%
\begin{align}
2\ell\left(  -s,a,\chi\right)   &  =-\chi\left(  -1\right)  \sum
\limits_{j=0}^{l}\left(  -1\right)  ^{j}\frac{(s)_{j}}{j!}E_{j,\overline{\chi
}}(0)a^{s-j}\nonumber\\
&  \quad+\frac{(s)_{l+1}}{l!}\int\limits_{0}^{\infty}\overline{E}%
_{l,\overline{\chi}}\left(  -t\right)  \left(  t+a\right)  ^{s-l-1}dt,
\label{27}%
\end{align}
where the integral converges absolutely for $\operatorname{Re}(s)<l$ and
represents an analytic function. Substituting (\ref{27}) in (\ref{23}) gives%
\begin{align}
2\ell_{s}\left(  x,a,\chi\right)   &  =\chi\left(  -1\right)  \sum
\limits_{j=0}^{l}\left(  -1\right)  ^{j}\frac{(s)_{j}}{j!}\overline
{E}_{j,\overline{\chi}}\left(  x\right)  \left(  x+a\right)  ^{s-j}%
+2\ell\left(  -s,a,\chi\right) \nonumber\\
&  \quad-\frac{(s)_{l+1}}{l!}\int\limits_{x}^{\infty}\overline{E}%
_{l,\overline{\chi}}\left(  -t\right)  \left(  t+a\right)  ^{s-l-1}dt,
\label{28}%
\end{align}
for $\operatorname{Re}(s)<l$ and $x\geq0.$

Writing $x=0$ in (\ref{28}) yields
\begin{equation}
2\ell\left(  -s,a,\chi\right)  =\sum\limits_{j=0}^{l}\frac{(s)_{j}}%
{j!}E_{j,\overline{\chi}}\left(  0\right)  a^{s-j}+\frac{(s)_{l+1}}{l!}%
\int\limits_{0}^{\infty}\frac{\overline{E}_{l,\overline{\chi}}\left(
-t\right)  }{\left(  t+a\right)  ^{l+1-s}}dt, \label{30}%
\end{equation}
which is valid for $\operatorname{Re}(s)<l$.

\section{Some consequences}

This section is concerned with some formulas about the alternating Dirichlet
$L$--function and counterparts of the Examples 6--10 of \cite{2}.

\subsection{Around the alternating Dirichlet $L$--function}

It is clear from (\ref{30}) that for $l=p$ and $s=p-1$ with $0<a<1,$
\[
2\ell\left(  1-p,a,\chi\right)  =\sum\limits_{j=0}^{p-1}\binom{p-1}%
{j}E_{j,\overline{\chi}}\left(  0\right)  a^{p-1-j}=E_{p-1,\overline{\chi}%
}\left(  a\right)  ,\text{ }p\geq1.
\]
Also, for $\operatorname{Re}(s)>0=l$
\[
2\ell\left(  s,a,\chi\right)  =-\chi\left(  -1\right)  E_{0,\overline{\chi}%
}\left(  0\right)  a^{-s}+\chi\left(  -1\right)  s\int\limits_{0}^{\infty
}\frac{\overline{E}_{0,\overline{\chi}}\left(  t\right)  }{\left(  t+a\right)
^{1+s}}dt
\]
and for $\operatorname{Re}(s)>-1,$ $\left(  l=1\right)  $%
\begin{equation}
2\ell\left(  s,a,\chi\right)  =E_{0,\overline{\chi}}\left(  0\right)
a^{-s}-sE_{1,\overline{\chi}}\left(  0\right)  a^{-s-1}+s\left(  s+1\right)
\chi\left(  -1\right)  \int\limits_{0}^{\infty}\frac{\overline{E}%
_{1,\overline{\chi}}\left(  t\right)  }{\left(  t+a\right)  ^{2+s}}dt.
\label{22}%
\end{equation}

Differentiating both sides of (\ref{22}) with respect to $s$ at $s=0$\ gives
\begin{align}
2\frac{d}{ds}\ell\left(  s,a,\chi\right)  |_{s=0}  &  =2\ell^{\prime}\left(
0,a,\chi\right) \nonumber\\
&  =-E_{0,\overline{\chi}}\left(  0\right)  \log a-\frac{1}{a}\overline
{E}_{1,\overline{\chi}}\left(  0\right)  +\chi\left(  -1\right)
\int\limits_{0}^{\infty}\frac{\overline{E}_{1,\overline{\chi}}\left(
x\right)  }{\left(  x+a\right)  ^{2}}dx. \label{1b}%
\end{align}

Similar results for $\ell\left(  s,\chi\right)  =\ell\left(  s,0,\chi\right)
$ can be achieved by applying Theorem \ref{C-B-S} to $f(x)=x^{-s},$
$\operatorname{Re}\left(  s\right)  >0,$ where $\alpha=1$ and $\beta=2kN$,
$N\in\mathbb{N}$. Following the arguments in the proof of (\ref{30}) and then
letting\textbf{\ }$N\rightarrow\infty$ give rise to%
\begin{align*}
2\ell\left(  s,\chi\right)  +2  &  =-\chi\left(  -1\right)  \sum_{j=1}%
^{l}\frac{s\left(  s+1\right)  ...\left(  s+j-1\right)  }{j!}\overline
{E}_{j,\overline{\chi}}\left(  1\right) \\
&  +\chi\left(  -1\right)  \frac{s\left(  s+1\right)  ...\left(  s+l\right)
}{l!}\int\limits_{1}^{\infty}\frac{\overline{E}_{l,\overline{\chi}}\left(
x\right)  }{x^{s+l+1}}dx,
\end{align*}
where the integral is analytic for $\operatorname{Re}\left(  s\right)  >-l$.
In particular, for $l=1,$
\[
2\ell\left(  s,\chi\right)  =-2-\chi\left(  -1\right)  \overline
{E}_{0,\overline{\chi}}\left(  1\right)  -\chi\left(  -1\right)  s\overline
{E}_{1,\overline{\chi}}\left(  1\right)  +\chi\left(  -1\right)  s\left(
s+1\right)  \int\limits_{1}^{\infty}\frac{\overline{E}_{1,\overline{\chi}%
}\left(  x\right)  }{x^{s+2}}dx.
\]
Differentiating both sides of the equality above with respect to $s$ at $s=0$
gives%
\begin{equation}
2\ell^{\prime}\left(  0,\chi\right)  =-\chi\left(  -1\right)  \overline
{E}_{1,\overline{\chi}}\left(  1\right)  +\chi\left(  -1\right)
\int\limits_{1}^{\infty}\frac{\overline{E}_{1,\overline{\chi}}\left(
x\right)  }{x^{2}}dx. \label{lf8}%
\end{equation}

Observe that the integrals in (\ref{1b}) and (\ref{lf8}) can be emerged from
Theorem \ref{C-B-S} by taking the logarithm function. So, we may establish a
connection between generalized Euler functions and some identities for
logarithmic means.

\begin{proposition}
As $t$ tends to $+\infty,$ we have the following asymptotic expansion,%
\[
2\sum_{1\leq n<t}\left(  -1\right)  ^{n}\chi\left(  n\right)  \log\left(
t/n\right)  \sim2\ell^{\prime}\left(  0,\chi\right)  +\ell\left(
0,\chi\right)  \log t+\chi\left(  -1\right)  \sum\limits_{j=1}^{\infty}%
\frac{\overline{E}_{j,\overline{\chi}}\left(  t\right)  }{jt^{j}}.
\]

\end{proposition}

\begin{proof}
Let $f(x)=\log\left(  t/x\right)  ,$ $\alpha=1$ and $\beta=t$ and $l=1$ in
Theorem \ref{C-B-S}. Then%
\begin{align}
&  2\chi\left(  -1\right)  \sum_{1<n<t}\left(  -1\right)  ^{n}\chi\left(
n\right)  \log\left(  t/n\right) \nonumber\\
&  \;=2\chi\left(  -1\right)  \sum_{1\leq n<t}\left(  -1\right)  ^{n}%
\chi\left(  n\right)  \log\left(  t/n\right)  +2\log t\nonumber\\
&  \;=-\overline{E}_{1,\overline{\chi}}\left(  1\right)  -\overline
{E}_{0,\overline{\chi}}\left(  1\right)  \log t+\frac{\overline{E}%
_{1,\overline{\chi}}\left(  t\right)  }{t}+\int\limits_{1}^{t}\frac
{\overline{E}_{1,\overline{\chi}}\left(  x\right)  }{x^{2}}dx\nonumber\\
&  \;=2\chi\left(  -1\right)  \ell^{\prime}\left(  0,\chi\right)
-\overline{E}_{0,\overline{\chi}}\left(  1\right)  \log t+\frac{\overline
{E}_{1,\overline{\chi}}\left(  t\right)  }{t}-\int\limits_{t}^{\infty}%
\frac{\overline{E}_{1,\overline{\chi}}\left(  x\right)  }{x^{2}}dx,
\label{lf6}%
\end{align}
where we have used (\ref{lf8}). Using that $\overline{E}_{0,\overline{\chi}%
}\left(  1\right)  =\overline{E}_{0,\overline{\chi}}\left(  0\right)
-2\chi\left(  -1\right)  $ and $\ell\left(  0,\chi\right)  =\overline
{E}_{0,\overline{\chi}}\left(  0\right)  ,$ and integrating by parts
repeatedly with the use of (\ref{lf6}), one arrives at the asymptotic formula.
\end{proof}

We now apply Theorem \ref{C-B-S} to the function $f(x)=\log\left(  x+a\right)
,$ $-\pi<\arg a<\pi,$ where $\alpha=0,$ $\beta=2kN,$ $N\in\mathbb{N}$, and
$l=1$ to obtain
\begin{align}
&  2\sum_{n=1}^{2kN}\left(  -1\right)  ^{n}\chi\left(  n\right)  \log\left(
n+a\right) \nonumber\\
&  =-\overline{E}_{0,\overline{\chi}}\left(  0\right)  \log\left(
2kN+a\right)  -\frac{\overline{E}_{1,\overline{\chi}}\left(  0\right)
}{2kN+a}\nonumber\\
&  \quad+\overline{E}_{0,\overline{\chi}}\left(  0\right)  \log a+\chi\left(
-1\right)  \frac{\overline{E}_{1,\overline{\chi}}\left(  0\right)  }{a}%
-\chi\left(  -1\right)  \int\limits_{0}^{2Nk}\frac{\overline{E}_{1,\overline
{\chi}}\left(  x\right)  }{\left(  x+a\right)  ^{2}}dx. \label{lf5}%
\end{align}
Gathering (\ref{lf6}) and (\ref{lf5}) for $t=2kN,$ $N\in\mathbb{N}$, then
letting $N\rightarrow\infty$ and using (\ref{lf8}), we find that
\begin{align}
&  -2\sum_{n=1}^{\infty}\left(  -1\right)  ^{n}\chi\left(  n\right)  \left(
\log\left(  n\right)  -\log\left(  n+a\right)  \right) \nonumber\\
&  \quad=2\ell^{\prime}\left(  0,\chi\right)  +\chi\left(  -1\right)
\frac{\overline{E}_{1,\overline{\chi}}\left(  0\right)  }{a}+\overline
{E}_{0,\overline{\chi}}\left(  0\right)  \log a-\chi\left(  -1\right)
\int\limits_{0}^{\infty}\frac{\overline{E}_{1,\overline{\chi}}\left(
x\right)  }{\left(  x+a\right)  ^{2}}dx. \label{StF1}%
\end{align}
Note that the sum in (\ref{StF1}) is reminiscent of the definition of
character analogue of the gamma function defined by Berndt \cite[Definition
4]{2}. This motivates us to make the following definition.

\begin{definition}
\label{gamma}Let $\chi$ be a real primitive character. Define
\[
\Gamma^{\ast}\left(  a,\chi\right)  =\prod_{n=1}^{\infty}\left(  \frac{n}%
{n+a}\right)  ^{\left(  -1\right)  ^{n}\chi\left(  n\right)  }.
\]

\end{definition}

In the light of this definition, (\ref{StF1}) becomes\textbf{\ }%
\begin{equation}
2\log\Gamma^{\ast}\left(  a,\chi\right)  =-\overline{E}_{0,\overline{\chi}%
}\left(  0\right)  \log a-2\ell^{\prime}\left(  0,\chi\right)  -\chi\left(
-1\right)  \frac{\overline{E}_{1,\overline{\chi}}\left(  0\right)  }{a}%
+\chi\left(  -1\right)  \int\limits_{0}^{\infty}\frac{\overline{E}%
_{1,\overline{\chi}}\left(  x\right)  }{\left(  x+a\right)  ^{2}}dx,
\label{lf4}%
\end{equation}
which shows that $\Gamma^{\ast}\left(  a,\chi\right)  $ is well defined and
analytic for $-\pi<\arg a<\pi.$

Combining (\ref{1b}) and (\ref{lf4}), we infer the \textit{Lerch's formula}
for $\ell\left(  s,a,\chi\right)  $ as%
\begin{equation}
\ell^{\prime}\left(  0,a,\chi\right)  =\log\Gamma^{\ast}\left(  a,\chi\right)
+\ell^{\prime}\left(  0,\chi\right)  , \label{35}%
\end{equation}
which is the character analogue of the familiar formula%
\[
\zeta^{\prime}\left(  0,z\right)  =\log\Gamma\left(  z\right)  +\zeta^{\prime
}\left(  0\right)  ,
\]
where $\Gamma\left(  z\right)  $ and $\zeta\left(  z\right)  $\ are the Euler
gamma and Riemann zeta functions, respectively.

Furthermore, in (\ref{lf4}), integrating by parts repeatedly in view of
(\ref{8b}), we arrive at the following asymptotic formula, the counterpart of
\cite[Proposition 5.3]{2}.

\begin{proposition}
[Stirling's formula for $\log\Gamma^{\ast}\left(  a,\chi\right)  $]%
\label{Stf}For $-\pi<\arg a<\pi,$ as $a$ tends to $\infty,$
\[
\log\Gamma^{\ast}\left(  a,\chi\right)  \sim-\frac{1}{2}\ell\left(
0,\chi\right)  \log a-\ell^{\prime}\left(  0,\chi\right)  -\frac{\chi\left(
-1\right)  }{2}\sum\limits_{j=1}^{\infty}\frac{\overline{E}_{j,\overline{\chi
}}\left(  0\right)  }{ja^{j}},
\]
where the principal branch of the logarithm is taken.\textbf{\ }
\end{proposition}

Next we write the integral in (\ref{1b}) as in the form\textbf{\ }%
\begin{align*}
\int\limits_{0}^{\infty}\frac{\overline{E}_{1,\overline{\chi}}\left(
x\right)  }{\left(  x+a\right)  ^{2}}dt  &  =\sum\limits_{n=0}^{\infty}%
\int\limits_{2nk}^{2\left(  n+1\right)  k}\frac{\overline{E}_{1,\overline
{\chi}}\left(  x\right)  }{\left(  x+a\right)  ^{2}}dx\\
&  =\frac{1}{\left(  2k\right)  ^{2}}\int\limits_{0}^{2k}\overline
{E}_{1,\overline{\chi}}\left(  t\right)  \sum\limits_{n=0}^{\infty}\left(
n+\frac{t+a}{2k}\right)  ^{-2}dt\\
&  =\frac{1}{\left(  2k\right)  ^{2}}\int\limits_{0}^{2k}\overline
{E}_{1,\overline{\chi}}\left(  t\right)  \zeta\left(  2,\frac{t+a}{2k}\right)
dt.
\end{align*}
So, Eq. (\ref{1b}) becomes
\begin{equation}
2\ell^{\prime}\left(  0,a,\chi\right)  =-\overline{E}_{0,\overline{\chi}%
}\left(  0\right)  \log a-\frac{1}{a}\overline{E}_{1,\overline{\chi}}\left(
0\right)  +\frac{\chi\left(  -1\right)  }{\left(  2k\right)  ^{2}}%
\int\limits_{0}^{2k}\overline{E}_{1,\overline{\chi}}\left(  t\right)
\zeta\left(  2,\frac{t+a}{2k}\right)  dt. \label{2}%
\end{equation}
Since
\begin{equation}
\frac{d^{2}}{dz^{2}}\log\Gamma\left(  z\right)  =\frac{d}{dz}\psi\left(
z\right)  =\zeta\left(  2,z\right)  , \label{turev}%
\end{equation}
where $\psi\left(  z\right)  $\ is the digamma function, the integral in
(\ref{2}) may be arisen from Theorem \ref{C-B-S} by setting $f(x)=\log
\Gamma\left(  \left(  x+a\right)  /2k\right)  ,$ $\alpha=0,$ $\beta=2k$ and
$l=1.$\ Under the circumstances,
\begin{align}
&  2\sum_{n=0}^{2k-1}\left(  -1\right)  ^{n}\chi\left(  n\right)  \log
\Gamma\left(  \frac{n+a}{2k}\right) \nonumber\\
&  =-\overline{E}_{0,\overline{\chi}}\left(  0\right)  \log\frac{a}{2k}%
-\frac{1}{a}\overline{E}_{1,\overline{\chi}}\left(  0\right)  +\frac
{\chi\left(  -1\right)  }{\left(  2k\right)  ^{2}}\int\limits_{0}%
^{2k}\overline{E}_{1,\overline{\chi}}\left(  x\right)  \zeta\left(
2,\frac{x+a}{2k}\right)  dx, \label{1a}%
\end{align}
where we have used that $\Gamma\left(  z+1\right)  =z\Gamma\left(  z\right)  $
and $\psi\left(  z+1\right)  -\psi\left(  z\right)  =1/z.$ Assembling
(\ref{2}) and (\ref{1a}), we have%
\begin{equation}
2\ell^{\prime}\left(  0,a,\chi\right)  =-\overline{E}_{0,\overline{\chi}%
}\left(  0\right)  \log\left(  2k\right)  +2\sum_{n=1}^{2k-1}\left(
-1\right)  ^{n}\chi\left(  n\right)  \log\Gamma\left(  \frac{n+a}{2k}\right)
. \label{3}%
\end{equation}

The following proposition shows that $\Gamma^{\ast}\left(  a,\chi\right)  $ is
a quotient of ordinary gamma functions.

\begin{proposition}
We have%
\[
\Gamma^{\ast}\left(  a,\chi\right)  =%
{\displaystyle\prod\limits_{n=1}^{2k-1}}
\left(  \frac{\Gamma\left(  \left(  n+a\right)  /2k\right)  }{\Gamma\left(
n/2k\right)  }\right)  ^{\left(  -1\right)  ^{n}\chi\left(  n\right)  }.
\]

\end{proposition}

\begin{proof}
From%
\[
\ell\left(  s,2k,\chi\right)  =\ell\left(  s,\chi\right)  -\sum_{n=1}%
^{2k-1}\left(  -1\right)  ^{n}\chi\left(  n\right)  n^{-s},
\]
it is seen that
\begin{equation}
\ell^{\prime}\left(  0,2k,\chi\right)  =\ell^{\prime}\left(  0,\chi\right)
-\sum_{n=1}^{2k-1}\left(  -1\right)  ^{n}\chi\left(  n\right)  \log n.
\label{33}%
\end{equation}
Setting $a=2k$ in (\ref{3}) and then comparing with (\ref{33}) give
\[
\ell^{\prime}\left(  0,\chi\right)  =-\frac{1}{2}\ell\left(  0,\chi\right)
\log\left(  2k\right)  +\sum_{n=1}^{2k-1}\left(  -1\right)  ^{n}\chi\left(
n\right)  \log\Gamma\left(  \frac{n}{2k}\right)  .
\]
Substituting this in (\ref{35}) and then combining with (\ref{3}) lead to
\begin{align}
\log\Gamma^{\ast}\left(  a,\chi\right)   &  =\sum_{n=1}^{2k-1}\left(
-1\right)  ^{n}\chi\left(  n\right)  \left(  \log\Gamma\left(  \frac{n+a}%
{2k}\right)  -\log\Gamma\left(  \frac{n}{2k}\right)  \right) \label{37}\\
&  =\sum_{n=1}^{2k-1}\left(  -1\right)  ^{n}\chi\left(  n\right)  \log\left(
\frac{\Gamma\left(  \left(  n+a\right)  /2k\right)  }{\Gamma\left(
n/2k\right)  }\right)  ,\nonumber
\end{align}
which is the desired result.
\end{proof}

Let us continue by differentiating both sides of (\ref{37}) with respect to
$a$. Then, we have\textbf{\ }%
\begin{equation}
\frac{d}{da}\log\Gamma^{\ast}\left(  a,\chi\right)  =\frac{1}{2k}\sum
_{n=1}^{2k-1}\left(  -1\right)  ^{n}\chi\left(  n\right)  \psi\left(
\frac{n+a}{2k}\right)  , \label{39}%
\end{equation}
by (\ref{turev}). For the convenience with (\ref{turev}),\ the right-hand side
of (\ref{39}) can be denoted by $\psi^{\ast}\left(  a,\chi\right)  ,$ i.e.,
\[
\psi^{\ast}\left(  a,\chi\right)  =\frac{1}{2k}\sum_{n=1}^{2k-1}\left(
-1\right)  ^{n}\chi\left(  n\right)  \psi\left(  \frac{n+a}{2k}\right)  .
\]
On the other hand, in the light of (\ref{35}), differentiating both sides of
(\ref{1b}) with respect to $a,$ and then comparing with (\ref{22}) for $s=1$,
we see that
\begin{equation}
\ell\left(  1,a,\chi\right)  =-\psi^{\ast}\left(  a,\chi\right)  =-\frac
{1}{2k}\sum_{n=1}^{2k-1}\left(  -1\right)  ^{n}\chi\left(  n\right)
\psi\left(  \frac{n+a}{2k}\right)  . \label{36}%
\end{equation}
In general, for $m\geq0$ we have
\begin{equation}
\frac{d^{m}}{da^{m}}\psi^{\ast}\left(  a,\chi\right)  =\left(  -1\right)
^{m+1}m!\ell\left(  m+1,a,\chi\right)  , \label{42}%
\end{equation}
which implies the following identity, viewed as the Taylor expansion of
$\ell\left(  s,a,\chi\right)  $ in the second variable $a$.

\begin{proposition}
\label{41a}For $|z|<1$ we have
\begin{equation}
\sum_{m=2}^{\infty}\ell\left(  m,a,\chi\right)  z^{m-1}=\psi^{\ast}\left(
a,\chi\right)  -\psi^{\ast}\left(  a-z,\chi\right)  \label{41}%
\end{equation}

\end{proposition}

\begin{proof}
The statement follows from the Taylor expansion of $\psi^{\ast}\left(
z,\chi\right)  $ at $z=a.$
\end{proof}

The character analogue of the Weierstrass product representation of
$\Gamma\left(  s\right)  $ can be derived from Definition \ref{gamma} and also
from Proposition \ref{41a}.

\begin{proposition}
\label{GEP}We have for all $s$%
\begin{equation}
\Gamma^{\ast}\left(  s,\chi\right)  =e^{-s\ell\left(  1,\chi\right)  }%
{\displaystyle\prod\limits_{n=1}^{\infty}}
\left[  \left(  1+s/n\right)  ^{-1}e^{s/n}\right]  ^{\left(  -1\right)
^{n}\chi\left(  n\right)  }, \label{45}%
\end{equation}
where the product converges uniformly on any compact set $S$ which avoids the
points $s=-n$, where $n$ is a positive integer and $\left(  -1\right)
^{n}\chi\left(  n\right)  =1$.
\end{proposition}

\begin{proof}
The proof from Definition \ref{gamma} is exactly like the proof of Berndt
\cite[Proposition 5.4]{2}, so we omit it.

For the proof via Proposition \ref{41a}, integrating (\ref{41}) from $0$ to
$s$, we see that%
\[
\sum_{m=2}^{\infty}\ell\left(  m,a,\chi\right)  \frac{s^{m}}{m}=\log
\Gamma^{\ast}\left(  a-s,\chi\right)  -\log\Gamma^{\ast}\left(  a,\chi\right)
+s\psi^{\ast}\left(  a,\chi\right)  .
\]
Taking $s\rightarrow-s$ and $a=0,$ we have%
\begin{align}
\sum_{m=2}^{\infty}\ell\left(  m,0,\chi\right)  \frac{\left(  -s\right)  ^{m}%
}{m}  &  =\log\Gamma^{\ast}\left(  s,\chi\right)  -\log\Gamma^{\ast}\left(
0,\chi\right)  -s\psi^{\ast}\left(  0,\chi\right) \nonumber\\
&  =\log\Gamma^{\ast}\left(  s,\chi\right)  +s\ell\left(  1,\chi\right)  .
\label{43}%
\end{align}
The left-hand side of (\ref{43}) is
\begin{equation}
\sum_{n=1}^{\infty}\left(  -1\right)  ^{n}\chi\left(  n\right)  \sum
_{m=2}^{\infty}\frac{1}{m}\left(  -\frac{s}{n}\right)  ^{m}=\sum_{n=1}%
^{\infty}\left(  -1\right)  ^{n}\chi\left(  n\right)  \left[  \frac{s}{n}%
-\log\left(  1+\frac{s}{n}\right)  \right]  , \label{47}%
\end{equation}
where we have used that
\[
\sum_{m=2}^{\infty}\frac{r^{m}}{m}=-r-\log\left(  1-r\right)  ,\text{ for
}\left\vert r\right\vert <1.
\]
Combining (\ref{43}) and (\ref{47}) gives (\ref{45}).
\end{proof}

Note that another consequence of (\ref{36}) with $\psi\left(  1-x\right)
-\psi\left(  x\right)  =\pi\cot\pi x$ is%
\begin{equation}
2\ell\left(  m+1,\chi\right)  =-\frac{\left(  -\pi/2k\right)  ^{m+1}}{m!}%
\sum_{n=1}^{2k-1}\left(  -1\right)  ^{n}\chi\left(  n\right)  \cot^{\left(
m\right)  }\left(  \frac{\pi n}{2k}\right)  ,\text{ }m\geq0, \label{48a}%
\end{equation}
when $\chi\left(  -1\right)  \left(  -1\right)  ^{m+1}=1.$ Indeed, it is easy
to see that for $0\leq a<1,$%
\begin{align}
&  \ell\left(  1,a,\chi\right)  -\chi\left(  -1\right)  \ell\left(
1,-a,\chi\right) \nonumber\\
&  =-\frac{1}{2k}\sum_{n=1}^{2k-1}\left(  -1\right)  ^{n}\chi\left(  n\right)
\left\{  \psi\left(  \frac{n+a}{2k}\right)  -\psi\left(  1-\frac{n+a}%
{2k}\right)  \right\} \nonumber\\
&  =\frac{\pi}{2k}\sum_{n=1}^{2k-1}\left(  -1\right)  ^{n}\chi\left(
n\right)  \cot\left(  \pi\frac{n+a}{2k}\right)  . \label{48b}%
\end{align}
Now (\ref{48a}) follows from (\ref{42}) and (\ref{48b}) for $\chi\left(
-1\right)  \left(  -1\right)  ^{m+1}=1$\ and $a=0.$ In particular,%
\begin{align*}
2\ell\left(  1,\chi\right)   &  =\frac{\pi}{2k}\sum_{n=1}^{2k-1}\left(
-1\right)  ^{n}\chi\left(  n\right)  \cot\left(  \frac{\pi n}{2k}\right)
,\text{ for odd }\chi,\\
2\ell\left(  2,\chi\right)   &  =\left(  \frac{\pi}{2k}\right)  ^{2}\sum
_{n=1}^{2k-1}\left(  -1\right)  ^{n}\frac{\chi\left(  n\right)  }{\sin
^{2}\left(  \frac{\pi n}{2k}\right)  },\text{ for even }\chi,\\
2\ell\left(  3,\chi\right)   &  =\left(  \frac{\pi}{2k}\right)  ^{3}\sum
_{n=1}^{2k-1}\left(  -1\right)  ^{n}\chi\left(  n\right)  \frac{\cos\left(
\frac{\pi n}{2k}\right)  }{\sin^{3}\left(  \frac{\pi n}{2k}\right)  },\text{
for odd }\chi,\\
2\ell\left(  4,\chi\right)   &  =\frac{1}{3}\left(  \frac{\pi}{2k}\right)
^{4}\sum_{n=1}^{2k-1}\left(  -1\right)  ^{n}\chi\left(  n\right)  \left(
\frac{2}{\sin^{4}\left(  \frac{\pi n}{2k}\right)  }+\frac{\cos\left(
\frac{\pi n}{k}\right)  }{\sin^{4}\left(  \frac{\pi n}{2k}\right)  }\right)
,\text{ for even }\chi,
\end{align*}
which are analogues of Eqs. (5.9)--(5.12) of Alkan \cite{4}. Such sums and
many ones can be found in \cite{27,28,30}.

\subsection{Counterparts of the Examples 6--10 of \cite{2}}

In this part, we constitute $f\left(  x\right)  $ in Theorem \ref{C-B-S} in
order to give some formulas, the counterparts of the Examples 6--10 of
\cite{2}.

$\bullet$ Let $f(x)=e^{xt},$ $\alpha=0$ and $\beta=k$. Then%
\[
2\sum_{n=0}^{k}\left(  -1\right)  ^{n}\chi\left(  n\right)  e^{nt}=\chi\left(
-1\right)  \sum\limits_{j=0}^{l}\left(  -1\right)  ^{j+1}\overline
{E}_{j,\overline{\chi}}\left(  0\right)  \frac{t^{j}}{j!}\left(
e^{kt}+1\right)  -R_{l},
\]
where%
\begin{align}
\left\vert R_{l}\right\vert  &  \leq\frac{\left\vert t^{l+1}\right\vert }%
{l!}\int\limits_{0}^{k}\left\vert \overline{E}_{l,\overline{\chi}}\left(
x\right)  e^{xt}\right\vert dx\nonumber\\
&  \leq4e^{kt}\frac{\left\vert t^{l+1}\right\vert }{\left(  \pi/k\right)
^{l+1}}\zeta\left(  l+1\right)  \rightarrow0\text{\ as }l\rightarrow
\infty\text{ for }\left\vert t\right\vert <\pi/k. \label{29}%
\end{align}
Thus, we have the generating function for the number $E_{j,\overline{\chi}%
}\left(  0\right)  $ as
\[
\sum_{n=0}^{k-1}\left(  -1\right)  ^{n}\chi\left(  n\right)  \frac{2e^{nt}%
}{e^{kt}+1}=\sum\limits_{j=0}^{\infty}E_{j,\overline{\chi}}\left(  0\right)
\frac{t^{j}}{j!}.
\]

$\bullet$ Let $f(x)=\cos(xt),$ $\alpha=0$ and $\beta=k$. It is obvious from
(\ref{8c}) that $\overline{E}_{j,\chi}\left(  0\right)  =0$ if $\chi$ and $j$
have the same parity. If $\chi$ is odd, then
\[
2\sum_{n=0}^{k-1}\left(  -1\right)  ^{n}\chi\left(  n\right)  \cos\left(
nt\right)  =\chi\left(  -1\right)  \sum\limits_{j=0}^{l}\frac{\left(
-1\right)  ^{2j+1}}{(2j)!}\overline{E}_{2j,\overline{\chi}}\left(  0\right)
\left(  \cos\left(  kt\right)  +1\right)  t^{2j}(-1)^{j}-R_{l}%
\]
where, as in (\ref{29}), $R_{l}$ tends to $0$ as $l\rightarrow\infty$ for
$\left\vert t\right\vert <\pi/k$. So, we have
\[
\frac{2\sum_{n=1}^{k-1}\left(  -1\right)  ^{n}\chi\left(  n\right)
\cos\left(  nt\right)  }{\cos\left(  kt\right)  +1}=\sum\limits_{j=0}^{\infty
}\left(  -1\right)  ^{j}E_{2j,\overline{\chi}}\left(  0\right)  \frac{t^{2j}%
}{(2j)!},\text{ for }\left\vert t\right\vert <\frac{\pi}{k}.
\]
If $\chi$ is even, then similarly
\[
\frac{2\sum_{n=1}^{k-1}\left(  -1\right)  ^{n}\chi\left(  n\right)
\cos\left(  nt\right)  }{\sin\left(  kt\right)  }=\sum\limits_{j=0}^{\infty
}\left(  -1\right)  ^{j}E_{2j+1,\overline{\chi}}\left(  0\right)
\frac{t^{2j+1}}{(2j+1)!},\text{ for }\left\vert t\right\vert <\frac{\pi}{k}.
\]

$\bullet$ Let $f(x)=\sin(xt),$ $\alpha=0$ and $\beta=k$. If $\chi$ is odd,
then for $\left\vert t\right\vert <\pi/k$%
\[
\frac{2\sum_{n=1}^{k-1}\left(  -1\right)  ^{n}\chi\left(  n\right)
\sin\left(  nt\right)  }{\sin\left(  kt\right)  }=\sum\limits_{j=0}^{\infty
}(-1)^{j}E_{2j,\overline{\chi}}\left(  0\right)  \frac{t^{2j}}{(2j)!}%
\]
and if $\chi$ is even%
\[
\frac{2\sum_{n=1}^{k-1}\left(  -1\right)  ^{n}\chi\left(  n\right)
\sin\left(  nt\right)  }{\cos\left(  kt\right)  +1}=\sum\limits_{j=0}^{\infty
}(-1)^{j}E_{2j+1,\overline{\chi}}\left(  0\right)  \frac{t^{2j+1}}{(2j+1)!}.
\]

$\bullet$ By the similar way, it can be seen that if $\chi$ is odd, then for
$\left\vert t\right\vert <\pi/k$
\begin{align*}
\frac{2\sum_{n=0}^{k}\left(  -1\right)  ^{n}\chi\left(  n\right)  \cosh\left(
nt\right)  }{\cosh(kt)+1}  &  =\frac{2\sum_{n=1}^{k-1}\left(  -1\right)
^{n}\chi\left(  n\right)  \sinh\left(  nt\right)  }{\sinh(kt)}\\
&  =\sum\limits_{j=0}^{\infty}E_{2j,\overline{\chi}}\left(  0\right)
\frac{t^{2j}}{(2j)!}%
\end{align*}
and if $\chi$ is even%
\begin{align*}
\frac{2\sum_{n=0}^{k}\left(  -1\right)  ^{n}\chi\left(  n\right)  \cosh\left(
nt\right)  }{\sinh(kt)}  &  =\frac{2\sum_{n=0}^{k}\left(  -1\right)  ^{n}%
\chi\left(  n\right)  \sinh\left(  nt\right)  }{\cosh(kt)+1}\\
&  =\sum\limits_{j=0}^{\infty}E_{2j+1,\overline{\chi}}\left(  0\right)
\frac{t^{2j+1}}{(2j+1)!}.
\end{align*}

\section{Proofs of reciprocity theorems}

\begin{proof}
[Proof of Theorem \ref{recip2}]Let $f(x)=\overline{E}_{p,\chi}\left(
xb/c\right)  ,$ $\alpha=0$ and $\beta=ck$ in Theorem \ref{BS}. By virtue of
(\ref{8b}), for $1\leq l\leq p,$ one has\textbf{\ }%
\begin{align*}
2\sum_{n=0}^{ck}\left(  -1\right)  ^{n}\overline{E}_{p,\chi}\left(  n\frac
{b}{c}\right)   &  =\sum\limits_{j=0}^{l-1}\frac{E_{j}\left(  0\right)  }%
{j!}\left(  \frac{b}{c}\right)  ^{j}\frac{p!}{(p-j)!}\left(  \left(
-1\right)  ^{ck-1}\overline{E}_{p-j,\chi}\left(  bk\right)  +\overline
{E}_{p-j,\chi}\left(  0\right)  ^{\text{\ }}\right) \\
&  \quad+\frac{p!}{(l-1)!(p-l)!}\left(  \frac{b}{c}\right)  ^{l}%
\int\limits_{0}^{ck}\overline{E}_{p-j,\chi}\left(  \frac{b}{c}x\right)
\overline{E}_{l-1}\left(  -x\right)  dx.
\end{align*}
For odd $b+c,$ with the use of (\ref{8d}), one can write%
\begin{align}
2\sum_{n=0}^{ck}\left(  -1\right)  ^{n}\overline{E}_{p,\chi}\left(  n\frac
{b}{c}\right)   &  =2\sum\limits_{j=0}^{l-1}\left(  \frac{b}{c}\right)
^{j}\binom{p}{j}E_{j}\left(  0\right)  \overline{E}_{p-j,\chi}\left(  0\right)
\label{4a}\\
&  \quad+l^{\text{\ }}\binom{p}{l}\left(  \frac{b}{c}\right)  ^{l}\left(
-1\right)  ^{l}c\int\limits_{0}^{k}\overline{E}_{p-l,\chi}\left(  bx\right)
\overline{E}_{l-1}\left(  cx\right)  dx.\nonumber
\end{align}
Now, let $f(x)=\overline{E}_{p}\left(  xc/b\right)  ,$ $\alpha=0$ and
$\beta=bk$ in Theorem \ref{C-B-S}. Using (\ref{8a}),
\begin{align*}
&  2\chi\left(  -1\right)  \sum_{n=0}^{bk}\left(  -1\right)  ^{n}\chi\left(
n\right)  \overline{E}_{p}\left(  n\frac{c}{b}\right) \\
&  =\sum\limits_{j=0}^{l}\frac{\left(  -1\right)  ^{j}}{j!}\left(  \frac{c}%
{b}\right)  ^{j}\frac{p!}{(p-j)!}\left(  \overline{E}_{j,\overline{\chi}%
}\left(  bk\right)  \overline{E}_{p-j}\left(  ck\right)  -\overline
{E}_{j,\overline{\chi}}(0)\overline{E}_{p-j}\left(  0\right)  ^{\text{\ }%
}\right) \\
&  \quad-\frac{(-1)^{l}}{l!}\frac{p!}{(p-l-1)!}\left(  \frac{c}{b}\right)
^{l+1}\int\limits_{0}^{bk}\overline{E}_{l,\overline{\chi}}\left(  x\right)
\overline{E}_{p-l-1}\left(  x\frac{c}{b}\right)  dx\\
&  =\sum\limits_{j=0}^{l}\left(  -1\right)  ^{j}\binom{p}{j}\left(  \frac
{c}{b}\right)  ^{j}\overline{E}_{j,\overline{\chi}}\left(  0\right)
\overline{E}_{p-j}\left(  0\right)  ((-1)^{b+c}-1)\\
&  \quad-(-1)^{l}p\binom{p-1}{l}\left(  \frac{c}{b}\right)  ^{l+1}%
b\int\limits_{0}^{k}\overline{E}_{l,\overline{\chi}}\left(  bx\right)
\overline{E}_{p-l-1}\left(  cx\right)  dx,
\end{align*}
for\textbf{\ }$0\leq l\leq p-2.$ Then, for odd $b+c,$ we have%
\begin{align}
2\chi\left(  -1\right)  \sum_{n=0}^{bk}\left(  -1\right)  ^{n}\chi\left(
n\right)  \overline{E}_{p}\left(  n\frac{c}{b}\right)   &  =2\sum
\limits_{j=0}^{l}\left(  -1\right)  ^{j+1}\binom{p}{j}\left(  \frac{c}%
{b}\right)  ^{j}\overline{E}_{j,\overline{\chi}}\left(  0\right)  \overline
{E}_{p-j}\left(  0\right) \label{4b}\\
&  -(-1)^{l}p\binom{p-1}{l}\left(  \frac{c}{b}\right)  ^{l+1}b\int%
\limits_{0}^{k}\overline{E}_{l,\overline{\chi}}\left(  bx\right)  \overline
{E}_{p-l-1}\left(  cx\right)  dx.\nonumber
\end{align}
Taking $\chi\rightarrow\overline{\chi}$ and $l=2$ in (\ref{4a}) leads to%
\begin{align}
S_{p}^{\left(  1\right)  }\left(  b,c:\overline{\chi}\right)   &  =2\sum
_{n=0}^{ck}\left(  -1\right)  ^{n}\overline{E}_{p,\overline{\chi}}\left(
n\frac{b}{c}\right) \nonumber\\
&  =2E_{0}\left(  0\right)  \overline{E}_{p,\overline{\chi}}\left(  0\right)
+2\frac{bp}{c}E_{1}\left(  0\right)  \overline{E}_{p-1,\overline{\chi}}\left(
0\right) \nonumber\\
&  \quad+p(p-1)\left(  \frac{b}{c}\right)  ^{2}c\int\limits_{0}^{k}%
\overline{E}_{p-2,\overline{\chi}}\left(  bx\right)  \overline{E}_{1}\left(
cx\right)  dx. \label{4c}%
\end{align}
Taking $l=p-2$ in (\ref{4b}) yields%
\begin{align}
S_{p}^{\left(  2\right)  }(c,b:\chi)  &  =2\sum_{n=1}^{bk}\left(  -1\right)
^{n}\chi\left(  n\right)  \overline{E}_{p}\left(  n\frac{c}{b}\right)
\nonumber\\
&  =2\chi\left(  -1\right)  \sum\limits_{j=0}^{p-2}\left(  -1\right)
^{j+1}\binom{p}{j}\left(  \frac{c}{b}\right)  ^{j}\overline{E}_{j,\overline
{\chi}}\left(  0\right)  \overline{E}_{p-j}\left(  0\right) \nonumber\\
&  \quad-(-1)^{p}\chi\left(  -1\right)  p(p-1)\left(  \frac{c}{b}\right)
^{p-1}b\int\limits_{0}^{k}\overline{E}_{p-2,\overline{\chi}}\left(  bx\right)
\overline{E}_{1}\left(  cx\right)  dx. \label{4d}%
\end{align}
Combining (\ref{4c}) and (\ref{4d}), one obtains that
\[
c^{p}S_{p}^{\left(  1\right)  }\left(  b,c:\overline{\chi}\right)  +b^{p}%
S_{p}^{\left(  2\right)  }\left(  c,b:\chi\right)  =2\sum\limits_{j=0}%
^{p}\binom{p}{j}c^{j}b^{p-j}\overline{E}_{j,\overline{\chi}}\left(  0\right)
\overline{E}_{p-j}\left(  0\right)  ,
\]
for odd $(b+c)$ and $(-1)^{p}\chi\left(  -1\right)  =1$.
\end{proof}

\begin{proof}
[Proof of Theorem \ref{recip1}]The definition of
\[
S_{p}\left(  b,c:\chi\right)  =\sum\limits_{n=1}^{ck}\chi\left(  n\right)
\overline{B}_{p,\overline{\chi}}\left(  \frac{b+ck}{2c}n\right)
\]
in this form is not convenient to prove reciprocity formula by aid of
Euler--MacLaurin or Boole summation formula. So, $S_{p}\left(  b,c:\chi
\right)  $ should be modified to apply summation formulas. For this, using
(\ref{31}) in the definition of $S_{p}\left(  b,c:\chi\right)  $, and then
\cite[Lemma 5.5]{8}, we see that%
\begin{align}
S_{p}\left(  b,c:\chi\right)   &  =2^{-p}\chi\left(  2\right)  \sum
\limits_{n=1}^{ck}\chi\left(  n\right)  \overline{B}_{p,\overline{\chi}%
}\left(  \frac{bn}{c}\right)  -p\frac{\chi\left(  2\right)  }{2^{p+1}}%
\sum\limits_{n=1}^{ck}\chi\left(  n\right)  \overline{E}_{p-1,\chi}\left(
\frac{bn}{c}+kn\right) \nonumber\\
&  =\frac{\chi\left(  2c\right)  \overline{\chi}\left(  -b\right)  }%
{2^{p}c^{p-1}}\left(  k^{p}-1\right)  \overline{B}_{p}\left(  0\right)
-p\frac{\chi\left(  2\right)  }{2^{p+1}}\sum\limits_{n=1}^{ck}\left(
-1\right)  ^{n}\chi\left(  n\right)  \overline{E}_{p-1,\chi}\left(  \frac
{bn}{c}\right)  . \label{20}%
\end{align}
Now let $f\left(  x\right)  =\overline{E}_{p-1,\chi}\left(  xb/c\right)  ,$
$\alpha=0$ and $\beta=ck$ in Theorem \ref{C-B-S}. Then, in the light of
(\ref{8b}), we can write%
\begin{align}
&  \sum_{n=0}^{ck}\left(  -1\right)  ^{n}\chi\left(  n\right)  \overline
{E}_{p-1,\chi}\left(  n\frac{b}{c}\right) \nonumber\\
&  \ =\frac{\chi\left(  -1\right)  }{2}\sum\limits_{j=0}^{l}\left(  -1\right)
^{j}\binom{p-1}{j}\left(  \frac{b}{c}\right)  ^{j}\left\{  \left(  \left(
-1\right)  ^{(b+c)}-1\right)  \overline{E}_{j,\overline{\chi}}\left(
0\right)  \overline{E}_{p-1-j,\chi}\left(  0\right)  \right\} \nonumber\\
&  \quad-\frac{\chi\left(  -1\right)  }{2}\left(  -1\right)  ^{l}%
(p-1)\binom{p-2}{l}\left(  \frac{b}{c}\right)  ^{l+1}\int\limits_{0}%
^{ck}\overline{E}_{l,\overline{\chi}}\left(  x\right)  \overline
{E}_{p-2-l,\chi}\left(  \frac{b}{c}x\right)  dx. \label{19}%
\end{align}
Following precisely the method in the proof of Theorem \ref{recip2} and using
that $\overline{B}_{p}\left(  0\right)  =0$ for odd $p$ yield%
\begin{align*}
&  \overline{\chi}\left(  -2\right)  bc^{p}S_{p}\left(  b,c:\chi\right)
+\chi\left(  -2\right)  cb^{p}S_{p}\left(  c,b:\overline{\chi}\right) \\
&  =\frac{p}{2^{p+1}}\sum\limits_{j=1}^{p}\left(  -1\right)  ^{j}\binom
{p-1}{j-1}c^{j}b^{p+1-j}\overline{E}_{j-1,\chi}\left(  0\right)  \overline
{E}_{p-j,\overline{\chi}}\left(  0\right)  .
\end{align*}

\end{proof}

\begin{remark}
Taking into consideration (\ref{31}), this formula coincides with
\cite[Corollary 4.3]{29} wherein there is the condition $b$ or $c\equiv
0\left(  \operatorname{mod}k\right)  .$
\end{remark}

We conclude the study with some results for the integral involving character
Euler functions in consequence\ of (\ref{19}) and (\ref{20}). We first note
that the sum on the left-hand side of (\ref{19}) is zero when $p$ and $\left(
b+c\right)  $ have opposite parity. Therefore, if $p>1$ is odd and $(b+c)$ is
even, then%
\[
\int\limits_{0}^{k}\overline{E}_{l,\overline{\chi}}\left(  x\right)
\overline{E}_{p-2-l,\chi}\left(  \frac{b}{c}x\right)  dx=0
\]
and if $p$ is even and $\left(  b+c\right)  $ is odd, then%
\begin{align*}
&  \int\limits_{0}^{k}\overline{E}_{l,\overline{\chi}}\left(  cx\right)
\overline{E}_{p-2-l,\chi}\left(  bx\right)  dx\\
&  =\frac{2\left(  -c/b\right)  ^{l+1}}{c\left(  p-1\right)  \binom{p-2}{l}%
}\sum\limits_{j=0}^{l}\left(  -1\right)  ^{j}\binom{p-1}{j}\left(  \frac{b}%
{c}\right)  ^{j}\overline{E}_{j,\overline{\chi}}\left(  0\right)  \overline
{E}_{p-1-j,\chi}\left(  0\right)  .
\end{align*}

Let $p$ and $\left(  b+c\right)  $ be even. Gathering $S_{p}(b,c:\chi
)=c^{1-p}\chi\left(  2c\right)  \overline{\chi}\left(  -b\right)  \left(
k^{p}-1\right)  \overline{B}_{p}\left(  0\right)  $ (\cite[Proposition
5.7]{29}) and (\ref{20}), one arrives%
\[
\sum\limits_{n=1}^{ck}\left(  -1\right)  ^{n}\chi\left(  n\right)
\overline{E}_{p-1,\chi}\left(  \frac{bn}{c}\right)  =\frac{1}{p}2\left(
1-2^{p}\right)  c^{1-p}\chi\left(  c\right)  \overline{\chi}\left(  -b\right)
\left(  k^{p}-1\right)  B_{p}.
\]

Thus, from the fact that $2\left(  2^{p}-1\right)  B_{p}=-pE_{p-1}\left(
0\right)  ,$ we have
\[
\int\limits_{0}^{k}\overline{E}_{l,\overline{\chi}}\left(  cx\right)
\overline{E}_{p-2-l,\chi}\left(  bx\right)  dx=2\left(  -1\right)  ^{l+1}%
\frac{\chi\left(  c\right)  \overline{\chi}\left(  b\right)  }{c^{p-l-1}%
b^{l+1}}\frac{\left(  k^{p}-1\right)  }{\binom{p-2}{l}}\frac{E_{p-1}\left(
0\right)  }{p-1}.
\]

\end{document}